\documentclass[smallextended]{svjour3}
\usepackage{latexsym,amssymb,amsmath}
\smartqed
\usepackage{graphicx}
\usepackage{color}
\RequirePackage{fix-cm}
\usepackage[section]{placeins}
\def\ox{\bar{x}}

\def\epi{{\rm epi}\,}

\begin{document}
	
	\renewcommand{\theequation}{\thesection.\arabic{equation}}
	\normalsize
	
	\setcounter{equation}{0}
	
	\title{Optimality conditions based on the Fr\'echet second-order subdifferential}
	
	\author{D.T.V.~An       \and
		N.D. Yen
	}

	\institute{D.T.V.~An\ \at
		Department of Mathematics and Informatics, Thai Nguyen University of Sciences, Thai Nguyen city, Vietnam\\
		\email{andtv@tnus.edu.vn}           	\and
		N.D. Yen
		 \at
	Institute of Mathematics, Vietnam Academy of Science and Technology, Hanoi, Vietnam\\
		\email{ndyen@math.ac.vn}
	}
	\date{Received: date / Accepted: date}
	
	\maketitle
	
	\begin{abstract}
     This paper focuses on second-order necessary optimality conditions for constrained optimization problems on Banach spaces. For problems in the classical setting, where the objective function is $C^2$-smooth, we show that strengthened second-order necessary optimality conditions are valid if the constraint set is generalized polyhedral convex. For problems in a new setting, where the objective function is just assumed to be $C^1$-smooth and the constraint set is generalized polyhedral convex, we establish sharp second-order necessary optimality conditions based on the Fr\'echet second-order subdifferential of the objective function and the second-order tangent set to the constraint set. Three examples are given to show that the used hypotheses are essential for the new theorems. Our second-order necessary optimality conditions refine and extend several existing results.
	\end{abstract}
	

	
	\newpage
	\section{Introduction}
\markboth{\centerline{\it Optimality conditions}}{\centerline{\it D.T.V.~An
		and N.D.~Yen}} \setcounter{equation}{0}
		
			It is well-known that second-order optimality conditions are fundamental results in nonlinear mathematical programming \cite{Ben-Tal1980,Ben-Tal1982,Bonnans_Shapiro_2000,L_Y_2008,McCormick1967,Penot1994,Penot1999,Polyak,Ruszczynski2006}, which have numerous applications in stability and sensitivity analysis, as well as in numerical methods for optimization problems. The need of generalizing these conditions to broader settings continues to attract attention of many researchers; see, e.g., \cite{ChieuLeeYen2017,HSN_1984,Huy_Tuyen} and the references therein.

			\medskip
			In classical second-order optimality conditions, the objective function of the finite-dimensional optimization problem in question is assumed to be twice continuously differentiable (a $C^2$-smooth function for short). If the  objective function is continuously Fr\'echet differentiable and the gradient mapping is locally Lipschitz, then one has deal with a $C^{1,1}$-smooth problem. Second-order optimality conditions for finite-dimensional $C^{1,1}$- smooth optimization problems have been obtained by Hiriart-Urruty \textit{et al.}~\cite{HSN_1984}, Huy and Tuyen~\cite{Huy_Tuyen}.

		 \medskip
		If the objective function of an optimization problem is continuously Fr\'echet differentiable and the gradient mapping is merely continuous, then one has deal with a $C^{1}$-smooth problem. The class of $C^{1}$-smooth optimization problems is much larger than that of $C^{1,1}$- smooth optimization problems. As far as we know, the tools employed in \cite{HSN_1984,Huy_Tuyen} are no longer suitable for $C^{1}$-smooth problems. To describe locally optimal solutions of $C^1$-smooth unconstrained minimization problems in a Banach space setting, Chieu \textit{et al.}~\cite{ChieuLeeYen2017} have explored the possibility of using the Fr\'echet second-order subdifferential and the limiting second-order subdifferential, which can be viewed as generalized Hessians of extended-real-valued functions. These concepts are due to Mordukhovich \cite{Mordukhovich_1992,Mordukhovich_2006a}. The limiting second-order subdifferential has many applications in stability analysis of optimization problems; see, e.g., \cite{Mo_Ro_SIOPT2012,MRS_SIOPT2013,Poli_Roc_1998} and the references therein. As shown in~\cite{ChieuChuongYaoYen2011,ChieuHuy2011}, the Fr\'echet second-order subdifferential is very useful in characterizing convexity of extended-real-valued functions. The authors of~\cite{ChieuLeeYen2017} have shown that the Fr\'echet second-order subdifferential is suitable for presenting second-order necessary optimality conditions~\cite[Theorems~3.1 and 3.3]{ChieuLeeYen2017}, while the limiting second-order subdifferential works well for second-order sufficient optimality conditions~\cite[Theorem~4.7 and Corollary~4.8]{ChieuLeeYen2017}. Consulting a preprint version of~\cite{ChieuLeeYen2017}, which appeared in 2013, Dai~\cite{LVD2014} has extended the finite-dimensional version of \cite[Theorem 3.3]{ChieuLeeYen2017} to the case of $C^1$-smooth optimization problems whose  constraint sets are described by linear equalities.

		\medskip
		Our interest in knowing deeper the role of second-order tangent sets in second-order optimality conditions mainly comes from the book of Bonnans and Shapiro~\cite{Bonnans_Shapiro_2000} and Theorem 3.45 in the book by Ruszczynski \cite{Ruszczynski2006}. When the second-order derivative of the $C^2$-smooth objective function is replaced by the Fr\'echet second-order subdifferential or the limiting second-order subdifferential, nontrivial questions arise if one wants to have second-order optimality conditions based on second-order tangent sets. Since optimization problems with polyhedral convex constraint sets or generalized polyhedral convex constraint sets will be encountered frequently in our investigations, we remark that they are of great importance in optimization theory (see for example~\cite{MRS_SIOPT2013}, where full stability of the local minimizers of such problems was characterized).  An extended-real-valued function defined on a Banach space is said to be a generalized polyhedral convex function if its epigraph is a generalized polyhedral convex set. The interested reader is referred to \cite[pp.~71--77]{Luan_Yen} and~\cite{Luan_Yao} for more comments on the role of generalized polyhedral convex sets and generalized polyhedral convex functions.
				
\medskip
The main goal of this paper is to clarify the applicability of the Fr\'echet second-order subdifferential to establishing second-order optimality conditions for constrained minimization problems. For problems in the classical setting, where the objective function is $C^2$-smooth, we show that strengthened second-order necessary optimality conditions are valid if the constraint set is generalized polyhedral convex. For problems in a new setting, where the objective function is just assumed to be $C^1$-smooth and the constraint set is generalized polyhedral convex, we establish sharp second-order necessary optimality conditions based on the Fr\'echet second-order subdifferential of the objective function and the second-order tangent set~to the constraint set. Our second-order necessary optimality conditions refine and extend several existing results. We will give three examples to show that the used hypotheses are essential for the new theorems.

\medskip			
The paper organization is as follows. Section~2 presents some basic definitions and auxiliary results. Section~3 is devoted to second-order optimality conditions for constrained optimization problems, where the objective function is $C^2$-smooth. Section 4 studies the possibility of using the Fr\'echet second-order subdifferential in second-order necessary optimality conditions for  constrained optimization problems, where the objective function is $C^1$-smooth.

\section{Preliminaries}
\markboth{\centerline{\it Optimality conditions}}{\centerline{\it D.T.V.~An
		and N.D.~Yen}} \setcounter{equation}{0}
	
	Let $X$ be a Banach space over the reals with the dual and the second dual being denoted, respectively, by $X^*$ and $X^{**}$. As usual, for a subset $\Omega \subset X$, we denote its \textit{convex hull} (resp., \textit{interior}, and \textit{boundary}) by ${\rm conv}\,\Omega$ (resp., ${\rm int}\Omega$, and $\partial\Omega$). One says that a nonempty subset $K \subset X$ is a cone if $tK \subset K$ for any $t>0.$ Following \cite{Luan_Yao_Yen}, we abbreviate the smallest convex cone containing $\Omega$ to cone\,$\Omega$. Then, ${\rm cone}\, \Omega=\{ tx \mid t>0, \,x \in {\rm conv}\,\Omega\}.$ The \textit{polar} to a cone $K\subset X$ is $K^*:=\{x^*\in X^* \mid \langle x^*, x\rangle \le 0, \ \forall x\in K\}$. If $A$ is a matrix, then we denote its transpose by $A^T$. The set of positive integers is denoted by $\mathbb N$.

\medskip
The forthcoming subsection recalls the definitions of contingent cone and second-order tangent set.

\subsection{Second-order tangent sets}

\begin{definition} (See, e.g., \cite[Definition~3.11]{Ruszczynski2006}) A direction $v$ is called \textit{tangent} to the set $C \subset X$ at a point $\ox\in C$ if there exist sequences of points $x_k \in C$ and scalar $\tau_k>0$, $k\in\mathbb N$, such that $\tau_k\rightarrow 0$ and $v=\lim\limits_{k\rightarrow \infty} \big[\tau_k^{-1}(x_k-\ox)\big].$
\end{definition}

The set of all tangent directions to $ C $ at a point $ \ox \in C $, denoted by $ T_{C}(\ox) $, is called the \textit{contingent cone} or the \textit{Bouligand-Severi tangent cone}~\cite[Chapter~1]{Mordukhovich_2006a} to $C$ at $\bar x$. From the definition it follows that $v\in  T_{C}(\ox) $ if and only if there exist a sequence $\{\tau_k\}$ of positive scalars and a sequence of vectors $\{v_k\}$ with $\tau_k\to 0$ and $v_k\to v$ as $k\to\infty$ such that $x_k:=\bar x+\tau_kv_k$ belongs to $C$ for all $k\in\mathbb N$. 

\begin{definition} (See, e.g., \cite[Definition~3.41]{Ruszczynski2006})
A vector $ w $ is called a \textit{second order tangent direction} to a set $ C \subset X$ at a point $ \ox \in C $ and in a tangent direction $v$, if there exist a sequence of scalars $ \tau_{k} > 0 $ and a sequence of points $ x^{k} \in C$ such that $ \tau_{k} \rightarrow 0 $ and
\begin{equation}\label{ttbac2}
w=\lim_{k\rightarrow \infty} \frac{x^{k}-\ox-\tau_k v }{\frac{\tau_k^2}{2}}.
\end{equation} 
\end{definition}

The set of all second-order tangent directions to $C$ at a point $ \ox \in C$ in a tangent direction $v$, denoted by $ T_{C}^{2}(\ox,v)$, is said to be the \textit{second-order tangent set} to $C$ at $\bar x$ in direction $v$. Note that the equality~(\ref{ttbac2}) can be rewritten as
\begin{equation*}x^{k}=\ox+\tau_k v+ \frac{\tau_k^2}{2}w +\mathnormal{o}(\tau_k^2).\end{equation*} Thus, $w\in  T_{C}^{2}(\ox,v) $ if and only if there exist a sequence $\{\tau_k\}$ of positive scalars and a sequence of vectors $\{w_k\}$ with $\tau_k\to 0$ and $w_k\to w$ as $k\to\infty$ such that $x_k:=\bar x+\tau_kv+\frac{\tau_k^2}{2}w_k$ belongs to $C$ for all $k\in\mathbb N$. 

\medskip
In the next subsection, we recall the definition of the generalized polyhedral convex set from~\cite{Bonnans_Shapiro_2000} and establish some auxiliary results.

\subsection{Generalized polyhedral convex sets}

\begin{definition}\label{Def_gpc}{\rm (See \cite[p. 133]{Bonnans_Shapiro_2000} and \cite[Definition~2.1]{Luan_Yao_Yen})} A subset $D\subset X$ is said to be a \textit{generalized polyhedral convex set} if there exist $x_i^*\in X^*$, $\alpha_i \in \mathbb{R},$ $ i=1,2,...,p$, and a closed affine subspace $L \subset X$, such that
	\begin{align}\label{generalized_polyhedra_convex}
	D=\{x\in X \mid x\in L, \ \langle x_i^*, x \rangle \le \alpha_i , \ i=1,2,...,p\}.
	\end{align} 
	If $D$ can be represented in the form of \eqref{generalized_polyhedra_convex} with $L=X$, then we say that it is a \textit{polyhedral convex set}.
\end{definition}

 From Definition \ref{Def_gpc} it follows that every generalized polyhedral convex set is a closed set. If $X$ is finite-dimensional, a subset $D\subset X$ is a generalized polyhedral convex set if and only if it is a polyhedral convex set; see \cite[p.~541]{Luan_Yao_Yen}. 
 
 \medskip
  Let $D$ be given as in \eqref{generalized_polyhedra_convex}. According to \cite[Remark 2.196]{Bonnans_Shapiro_2000}, there exists a continuous surjective linear mapping $A$ from $X$ to a Banach space $Y$ and a vector $y\in Y$ such that $L=\{x\in X \mid Ax=y\}$. Hence,
 \begin{align}\label{generalized_polyhedra_convex_1}
 D=\big\{x\in X \mid Ax=y, \ \langle x_i^*, x \rangle \le \alpha_i , \ i=1,2,...,p\big\}.
 \end{align} 
Put $I=\{1,2,..,p\}$ and, for any $x\in D$, let $I(x):=\{i\in I \mid \langle x_i^*,x \rangle =\alpha_i\}$.
 
  \medskip The first assertion of the next proposition can be found in \cite{Ban_Mordukhovich_Song_2011}. The second assertion extends the result in~\cite[Lemma~3.43]{Ruszczynski2006} to an infinite-dimensional spaces setting.
  
 \begin{proposition}\label{second-order-tangent-polyhedral_1}
 	Let $D$ be a generalized polyhedral convex set in a Banach space~$X$. The contingent cones and the second-order tangent sets to $D$ are represented as follows:
 	\begin{itemize}
 		\item[{\rm (i)}] $T_D(\bar x)=\{v\in X \mid Av=0,\; \langle x_i^*, v \rangle \le 0 ,\; i\in I(\bar x)\}$ for any $\bar x\in D$;
 		\item[{\rm (ii)}] $T^2_D(\ox , v)=T_{T_D(\ox)}(v)$ for any $\bar x\in D$ and $v\in T_D(\bar x)$.
 	\end{itemize}
   \end{proposition}
 \begin{proof}
 	(i) To show that
 	\begin{align}
 	\label{inclusion1}T_D(\bar x)\subset \{v\in X \mid Av=0, \ \langle x_i^*, v \rangle \le 0 , \ i\in I(\bar x)\},
 	\end{align}
 take any $v \in T_D(\ox)$. Let $\tau_{k} \downarrow 0$ and $v_k \rightarrow v$ be such that $\ox + \tau_{k} v_k \in D$ for $k \in \mathbb{N}$. Then, we have 
 	$A(\ox +\tau_{k} v_k)=y$ and $\langle x_i^*, \ox +\tau_{k} v_k \rangle \le \alpha_i$ for all $i\in I.$ This implies that
 	\begin{align}\label{formula1}
 	A(\tau_{k} v_k)=0 \ \; \mbox{and} \ \langle x_i^*,\tau_{k} v_k \rangle \le 0\ \; (\forall i\in I(\ox),\ \forall k\in \mathbb{N}).
 	\end{align}
 From \eqref{formula1} we have
 	\begin{align}\label{formula1a}
 	A(v_k)=0\ \; \mbox{and} \ \; \langle x_i^*, v_k \rangle \le 0\ \; (\forall i\in I(\ox),\ \forall k\in \mathbb{N}).
 	\end{align}
 	Letting $k\rightarrow \infty$, from~\eqref{formula1a} we get
 	$A(v)=0$ and $\langle x_i^*, v \rangle \le 0$ for any $i\in I(\ox)$.
 	In other words, $v$ belongs to the right-hand-side of~\eqref{inclusion1}. So, the inclusion~\eqref{inclusion1} is valid. To prove the opposite inclusion, pick any $v\in X$ satisfying $Av=0$ and $\langle x_i^* , v \rangle \le 0$ for $i\in I(\ox).$
 	Since $\ox \in D$, one has $A\ox=y$, $\langle x_i^* ,  \ox \rangle =\alpha_i$ for $i\in I(\ox)$, and $\langle x_i^*, \ox \rangle < \alpha_i$ for $i\in I\setminus I(\ox) .$ Hence, for all $t>0$ small enough, one has $A(\ox+t v)=y, \ \langle x_i^* ,  \ox+tv \rangle \le\alpha_i$ for $i\in I(\ox)$ and $\langle x_i^*, \ox+tv  \rangle < \alpha_i$ for $i\in I\setminus I(\ox) .$ So, $\ox +tv \in D$ for all $t>0$ small enough. It follows that $v\in T_D(\ox)$. Thus, assertion (i) is justified.
 			
 			 (ii) Fix any $\bar x\in D$ and $v\in T_D(\bar x)$. By assertion (i), $Av=0$ and $\langle x_i^*, v \rangle \le 0$ for all $i\in I(\bar x)$. Moreover, since 
 			 \begin{align}\label{tangent_new} T_D(\bar x)=\{u\in X \mid Au=0,\; \langle x_i^*, u \rangle \le 0 ,\; i\in I(\bar x)\},\end{align} applying the same assertion we can compute the contingent cone to the generalized polyhedral convex set $T_D(\ox)$ at $v$ as follows
 			\begin{align}\label{tangent_tangent}
 			T_{T_D(\ox)}(v)=\big\{u\in X \mid Au=0, \ \langle x_i^*, u\rangle \le 0, \ i\in I^0(v)\big\},
 			\end{align}
 			where $ I^0(v) :=\{i\in I(\ox) \mid \langle x_i^*, v \rangle =0\}$.
 		On one hand, for any fixed vector $w\in T^2_D(\ox , v)$, we can find sequences $\tau_{k} \downarrow 0$ and $w_k\rightarrow w$ such that $$\ox + \tau_{k}v+\frac{\tau^2_{k}}{2} w_k \in D\quad (\forall k\in \mathbb{N}).$$ By~\eqref{generalized_polyhedra_convex_1}, one has
 			$
 			A(\ox +\tau_{k} v+\frac{\tau^2_{k}}{2} w_k )=y $ and $ \langle x_i^*, \ox+\tau_{k}v+ \frac{\tau^2
 				_{k}}{2} w_k\rangle \le \alpha_i, \ i\in I.
 		$
 		As $\ox \in D$ and $v\in T_D(\ox)$, this yields
 			\begin{align}\label{formula1b}
 			A\Big(\frac{\tau^2_{k}}{2} w_k \Big)=0\ \mbox{and}\ \big\langle x_i^*, \frac{\tau^2
 				_{k}}{2} w_k\big\rangle \le 0,\ \forall i\in I^0(v).
 			\end{align}
 			Since $\tau_{k} >0$, \eqref{formula1b} implies that
 		$A\left(w_k \right)=0$  and $ \langle x_i^*, w_k\rangle \le 0$ for all $i\in I^0(v).$
 			Letting $k\rightarrow \infty$, we obtain $A\left(w\right)=0$  and $ \langle x_i^*, w\rangle \le 0$ for all $i\in I^0(v).$ Therefore, by \eqref{tangent_tangent} we can assert that $w\in T_{T_D(\ox)}(v).$
 	On the other hand, taking any $w\in T_{T_D(\ox)}(v)$, from~\eqref{tangent_tangent} one gets $Aw=0$ and $\langle x_i^*, w \rangle \le 0$ for all $i\in I^0(v).$ By the definition of  $I^0(v)$, we have $\langle x_i^*, v \rangle = 0$ for any $i\in I^0(v)$ and  $\langle x_i^*, v \rangle < 0$ for any $i\in I(\ox) \setminus I^0(v)$. Moreover, since $\ox \in D$, it holds that $A\ox=y$, $\langle x_i^* ,  \ox \rangle =\alpha_i$ for $i\in I(\ox)$, and $\langle x_i^*, \ox \rangle < \alpha_i$ for $i\in I\setminus I(\ox).$ So, for every $t>0$ sufficiently small, one has $A(\ox+t v+\frac{t^2}{2}w)=y, \ \langle x_i^* ,  \ox+tv +\frac{t^2}{2}w \rangle \le\alpha_i$ for all $i\in I^0(v)$ and $\langle x_i^*, \ox+tv +\frac{t^2}{2}w  \rangle < \alpha_i$ for all $i\in I\setminus I^0(v).$ This yields $\ox+tv +\frac{t^2}{2}w \in D$ for every $t>0$ sufficiently small. Hence, $w\in T^2_D(\ox , v).$
 	We have thus proved the equality stated in assertion (ii). $\hfill\Box$
 \end{proof}

\begin{remark}\label{remark_1} {\rm  If $D\subset X$ is a generalized polyhedral convex set then, for any $\bar x\in D$ and $v\in T_D(\bar x)$, one has $T_D(\bar x)\subset T^2_D(\ox,v)$, and the inclusion can be strict. We can justify this observation by representing $D$ in the form~\eqref{generalized_polyhedra_convex_1} and applying some formulas established in the proof of Proposition~\ref{second-order-tangent-polyhedral_1}. Indeed, since $I^0(v)\subset I(\ox)$, from~\eqref{tangent_new}, \eqref{tangent_tangent}, and the equality $T^2_D(\ox , v)=T_{T_D(\ox)}(v)$, one can deduce that $T_D(\bar x)\subset T^2_D(\ox , v)$. When $I^0(v)$ is a proper subset of $I(\ox)$, the last  inclusion can be strict. To have an example, one can choose $$D=\big\{x=(x_1,x_2)\in\mathbb R^2\mid x_1\geq 0, x_2\geq 0\big\},$$ $\bar x=(0,0)$, $v=(1,0)$, then use~\eqref{tangent_tangent} and the equality $T^2_D(\ox , v)=T_{T_D(\ox)}(v)$ to show that $T^2_D(\ox , v)=\big\{w=(w_1,w_2)\in\mathbb R^2\mid w_2\geq 0\big\}$, while $$T_D(\bar x)=\big\{u=(u_1,u_2)\in\mathbb R^2\mid u_1\geq 0, u_2\geq 0\big\}.$$}
\end{remark}

As a preparation for getting optimality conditions based on the Fr\'echet second-order subdifferential, we now recall the later concept and some related constructions.
 
\subsection{Constructions from generalized differentiation}

\begin{definition} {\rm (See \cite[p. 4 ]{Mordukhovich_2006a})}
	\rm
	Let $\Omega$ be a nonempty subset of $X.$ The \textit{Fr\'echet normal cone} to $\Omega$ at $x\in\Omega$ is given by
		\begin{align*}
	\widehat N_\Omega(x):=\Big\{ x^*\in X^*\mid \limsup\limits_{u \xrightarrow{\Omega}x} \dfrac{\langle x^*, u-x \rangle}{\|u-x\|} \leq 0 \Big\},
	\end{align*}
where $u \xrightarrow{\Omega} x$ means that $u \rightarrow x$ and $ u\in \Omega$.
	 If $x \not\in \Omega$, we put $\widehat N_\Omega(x)=\emptyset$.
\end{definition}

If $\Omega$ is convex, one has 
$$\widehat N_\Omega(x)=N_\Omega(x):=\big\{x^* \in X^* \mid \langle x^*, u-x \rangle \le 0, \ \forall u\in\Omega\big\},$$ i.e., $\widehat N_\Omega(x)$ coincides with the normal cone in the sense of convex analysis. In that case,  $ [T_\Omega(x)]^*=N_\Omega(x)$ and $[N_\Omega(x)]^*= T_\Omega(x)$, where $$[N_\Omega(x)]^*:=\{x\in X\mid \langle x^*,x \rangle \le 0,\ \forall x^*\in N_\Omega(x)\}.$$

\medskip
Given  a set-valued map  $F: X \rightrightarrows Y$ between Banach spaces, one defines  the \textit{graph} of $F$ by ${\rm{gph}}\, F=\{ (x,y) \in X \times Y \mid y \in F(x)\}.$ The product space $X\times Y$ is equipped with the norm $\|(x,y)\|:=\|x\|+\|y\|$.
	
	\begin{definition}{\rm (See \cite[p.~40]{Mordukhovich_2006a})
		 The \textit{Fr\'{e}chet coderivative} of $F$ at $\bar z=(\bar x, \bar y)$ in ${\rm{gph}}\, F$ is the multifunction $\widehat D^* F(\bar x, \bar y): Y^* \rightrightarrows X^*$ given by
			\begin{align*}
			\widehat D^* F(\bar z)(y^*)=\!\left\{x^* \in X^* \mid (x^*, -y^*) \!\in\! \widehat N_{{\rm gph}\, F}(\bar z)\right\}, \, \forall y^* \in Y^*.
			\end{align*}
			If $(\bar x, \bar y) \notin {\rm{gph}}\, F$, one puts $\widehat D^* F(\bar z)(y^*)=\emptyset$ for any $y^* \in Y^*$.
		}
	\end{definition}

If $F(x)=\{f(x)\}$ for all $x\in X$, where $f:X\to Y$ is a single-valued map, we will write $\widehat{D} f(\ox)(y^*)$ instead of $\widehat D^* F(\bar x,f(\bar x))(y^*)$.
 
\begin{proposition} {\rm (See \cite[Theorem 1.38]{Mordukhovich_2006a})} \label{coderivative_singleton}
	Let $f: X \rightarrow Y$ be a Fr\'echet differentiable function at $\ox$. Then $ \widehat{D} f(\ox)(y^*)=\{\nabla f(\ox) ^* y^*\}$ for every $y^*\in Y^*,$
	where $\nabla f(\ox) ^*$ is the adjoint operator of $\nabla f(\ox).$
\end{proposition} 

	
	Consider a function $f: X\rightarrow \overline{\Bbb{R}}$, where 
	$\overline{\Bbb{R}}=[- \infty, + \infty]$ is the extended real line. The \textit{epigraph} of $f$ is given by $ {\rm{epi}}\, f=\{ (x, \alpha) \in X \times \Bbb{R} \mid \alpha \ge f(x)\}.$ 
	
	\begin{definition} {\rm (See \cite[Chapter~1]{Mordukhovich_2006a}})
		\rm Let $f: X\rightarrow \overline{\mathbb{R}}$ be a  function defined on a Banach space.
		Suppose that $\bar {x} \in X$ and $ |f(\bar {x})| < \infty.$ One calls the set
		\begin{align*}
			\widehat \partial f(\bar x):=\left\{x^* \in X^*\mid (x^*, -1) \in \widehat N_{\epi f}( (\bar {x}, f(\bar{ x})))\right\}
		\end{align*}
	the \textit{Fr\'{e}chet subdifferential} of $f$ at $\bar {x}$. If $ |f(\bar x)| = \infty$, one puts $\widehat\partial f(\bar x)=\emptyset$.
	\end{definition}

	\begin{definition}  {\rm (See \cite[p.~122]{Mordukhovich_2006a})} {\rm Let $f:X\rightarrow \overline{\mathbb{R}}$
			be a function with a finite value at $\bar{x}.$ For any $\bar y\in \widehat\partial f (\bar x)$, the map $\widehat\partial^2 f(\bar x,\bar y):
		X^{**}\rightrightarrows X^* $ with the values 
			\begin{align*}
				\widehat{\partial}^2 f(\bar x,\bar y)(u):=(\widehat{D}^* \widehat{\partial}f)(\bar x,\bar y)(u)\quad (u \in X^{**})
			\end{align*} is said to be the \textit{Fr\'echet
				second-order subdifferential} of $f$ at $\bar x$ relative to
			$\bar y.$ }
	\end{definition}  

If $\widehat\partial f (\bar x)$ is a singleton, the symbol $\bar y$ in the notation $\widehat{\partial}^2 f(\bar x,\bar y)(u)$ will be omitted.
If $f:X\rightarrow \overline{\mathbb{R}}$ is Fr\'echet differentiable in an open neighborhood of $\bar x$, then $\widehat\partial  f (\bar x)=\{\nabla f(\ox)\}$.  Moreover, if the operator $\nabla f: X \rightarrow X^*$ is Fr\'echet  differentiable at $\ox$ with the second-order derivative $\nabla^2 f(\ox):=\nabla (\nabla f(\cdot))(\ox)$, then $\nabla^2 f(\ox)$ maps $X^{**}$ to $X^*$. By Proposition~\ref{coderivative_singleton}, $\widehat{\partial}^2 f(\bar x)(u)=\{\nabla^2 f(\ox)^* u\}$ for every $u \in X^{**}$. When $X$ is finite-dimensional and $f$ is $C^2$-smooth in an open neighborhood of $\bar x$, then $\nabla^2 f(\ox)$ is identified with the Hessian matrix of $f$ at $\bar x$ for which one has  $\nabla^2 f(\ox)^*=\nabla^2 f(\ox) $ by Clairaut's rule. 

\medskip
The forthcoming subsection presents two lemmas which will be used repeatedly in the sequel.

\subsection{Auxiliary results}

\begin{lemma}
	Let $C=\big\{x\in X \mid Ax=y, \ \langle x_i^*, x \rangle \le \alpha_i , \ i=1,2,...,p\big\},$ where $A$, $y$, $x_i^*,$ and $\alpha_i$ for $i=1,\dots,p$ are the same as in \eqref{generalized_polyhedra_convex_1}, be a generalized polyhedral convex set. For any $v\in T_C(\ox)$ with $-v \in T_C(\ox)$, it holds that \begin{align}\label{equality}
	T^2_C(\ox , -v)=T^2_C(\ox , v).
	\end{align}
\end{lemma}
\begin{proof}
	By Proposition~\ref{second-order-tangent-polyhedral_1}, $T_C^2(\ox ,v)=T_{T_C(\ox)}(v)$ and $T_C^2(\ox ,-v)=T_{T_C(\ox)}(-v).$ Moreover, one has $ T_{T_C(\ox)}(v)=[N_{T_C(\ox)}(v)]^*$ and $ T_{T_C(\ox)}(-v)=[N_{T_C(\ox)}(-v)]^*.$ Therefore, 
	\begin{align}\label{equalities}
     T_C^2(\ox ,v)=[N_{T_C(\ox)}(v)]^*\quad {\rm and}\quad T_C^2(\ox ,-v)=[N_{T_C(\ox)}(-v)]^*.
	\end{align}
	 On one hand, by~\cite[Proposition~4.2]{Luan_Yao_Yen}, $N_C(\ox) ={\rm cone}\,\big\{x_i^* \mid i \in I(\ox)\big\}+({\rm ker}\, A)^\intercal,$ where $I(\ox)=\{i\in I \mid \langle x_i^*, \ox \rangle =\alpha_i\}$ and $$({\rm ker}\, A)^\intercal =\{x^*\in X^* \mid \langle x^*, x \rangle =0, \ \forall x \in {\rm ker}\, A\}.$$
	On the other hand, according to Proposition~\ref{second-order-tangent-polyhedral_1},
	\begin{align*}
	T_C(\ox)=\{v\in X \mid Av=0, \ \langle x_i^*, v \rangle \le 0 , \ i\in I(\bar x)\}.
	\end{align*}
	So, $v\in T_C(\ox)$ and $-v\in T_C(\ox)$ if and only if $Av=0,$ $\langle x_i^*, v \rangle \le 0,$ and $ \langle x_i^*, -v \rangle \le 0$ for all $i\in I(\ox).$
	This means that $Av=0$ and $\langle x_i^*, v \rangle =0$ for all $i\in I(\ox).$ Putting $I^0(u) =\{i\in I(\ox) \mid \langle x_i^*, u \rangle =0\}$ for every  $u\in T_C(\ox),$ we see that $I^0(v)=I(\ox)=I^0(-v)$. So, thanks to \cite[Proposition~4.2]{Luan_Yao_Yen}, we have
	\begin{align*}
     N_{T_C(\ox)}(v) ={\rm cone}\,\{x_i^* \mid i \in I^0(v)\}+({\rm ker}\, A)^\intercal
	\end{align*}
	and $N_{T_C(\ox)}(-v) ={\rm cone}\,\{x_i^* \mid i \in I^0(v)\}+({\rm ker}\, A)^\intercal$. Thus, by~\eqref{equalities} we get
	\begin{align*}
	T_C^2(\ox , -v)= [N_{T_C(\ox)}(-v)]^*
	=[N_{T_C(\ox)}(v)]^*
	=	T_C^2(\ox , v).
	\end{align*} This justifies~\eqref{equality} and completes the proof.
	$\hfill\Box$
\end{proof}

Consider the problem 
\begin{equation*}
\min \{f(x) \mid x \in C\}, \tag{\rm P}
\end{equation*}
where $ f: X\rightarrow \mathbb{R} $ is a Fr\'echet differentiable function and $C$ is a nonempty subset of $X$. 

\begin{lemma}\label{Lemma1}
	Suppose that $\ox$ is a local minimum of {\normalfont(P)}, where $C$ is a generalized polyhedral convex set. Then,  $\langle \nabla f(\bar x) , v \rangle \ge 0$ for every $v \in T_C(\bar x).$ 
	Moreover, if $v\in T_C(\bar x)$ is such that $\langle \nabla f(\bar x), v\rangle =0$, then
	\begin{align}\label{formula_new}
	\langle \nabla f(\bar x), w \rangle \ge 0 \ \; \mbox{for all}\ \, w \in T_C^2(\bar x, v).
	\end{align}
\end{lemma}
\noindent\proof\ 
The first assertion is a special case of the result recalled in Theorem~\ref{necessary_condition} below. Let $v\in T_C(\bar x)$ be such that $\langle \nabla f(\bar x), v\rangle =0$. To get~\eqref{formula_new}, fix any 
 $w\in T_C^2(\bar x, v)$. By Proposition \ref{second-order-tangent-polyhedral_1} we have $T_C^2(\bar x, v)=T_{T_C(\bar x)}(v).$ Moreover, since $C$ is a generalized polyhedral convex set, $T_C(\ox)$ is a generalized polyhedral convex cone by~\cite[Proposition~2.22]{Luan_Yao_Yen}. So, applying~\cite[Proposition~2.22]{Luan_Yao_Yen}, one has
$T_{T_C(\bar x)}(v)={\rm cone}\,(T_C(\ox)-v).$
Thus, the representation $w=\lambda(v'-v)$ holds for some $v'\in T_C(\ox)$ and $\lambda >0$. Therefore, $$\langle \nabla f(\bar x), w \rangle = \lambda \langle \nabla f(\bar x),v'\rangle - \lambda \langle \nabla f(\bar x),v\rangle.$$ As $\langle \nabla f(\bar x), v'\rangle \ge0$ for any $v'\in T_C(\ox)$ by the first assertion and $\langle \nabla f(\bar x), v\rangle =0$ by our assumption, this implies~\eqref{formula_new}.
$\hfill\Box$

\section{Problems in the classical setting}
\markboth{\centerline{\it Optimality conditions}}{\centerline{\it D.T.V.~An
		and N.D.~Yen}} \setcounter{equation}{0}
In this section, we focus on second-order optimality conditions for problem (P) under the assumption that \textit{$f$ is twice continuously differentiable} on $X$ (i.e., $f$ is a $C^2$-smooth function). By abuse of terminology, we call this (P) \textit{a problem in the classical setting}.

\medskip
The next first-order and second-order necessary optimality conditions are known results. The proofs in a finite-dimensional setting given in \cite[p.~114 and p.~144]{Ruszczynski2006} are also valid for the infinite-dimensional setting adopted in the present paper. For the first statement, it suffices to assume that $f$ is Fr\'echet differentiable at $\bar x$.
 
\begin{theorem}\label{necessary_condition} {\rm (See, e.g., \cite[Theorem 3.24]{Ruszczynski2006})} If $\bar x$ is a local minimum of {\normalfont(P)}, then
	\begin{align}\label{Fermat_Rule}
		\langle \nabla f(\ox), v   \rangle \ge 0\ \; \mbox{for all}\ \, v\in T_C(\ox).
	\end{align}	
\end{theorem}

\begin{theorem}\label{differentiable_case_theorem}{\rm (See, e.g., \cite[Theorem 3.45]{Ruszczynski2006})} Assume that $\bar x$ is a local minimum of~{\normalfont(P)}. Then \eqref{Fermat_Rule} holds and, for every $v\in T_C(\bar x)$ satisfying $\langle \nabla f(\bar x), v\rangle =0$, one has
\begin{align}\label{differentiable_case}
\langle \nabla f(\bar x), w \rangle + \langle \nabla^2 f(\bar x) v,v \rangle \ge 0\ \; \mbox{for all}\ \, w \in T_C^2(\bar x, v).
\end{align}
\end{theorem}

Clearly, the simultaneous fulfillment of the inequalities $\langle \nabla f(\bar x), w \rangle\geq 0$ and $\langle \nabla^2 f(\bar x) v,v \rangle \ge 0$ yields the inequality $\langle \nabla f(\bar x), w \rangle + \langle \nabla^2 f(\bar x) v,v \rangle \ge 0$ in~\eqref{differentiable_case}. Hence, it is reasonable to raise the next question.

\medskip
\noindent {\sc Question 1:} \textit{When Theorem~\ref{differentiable_case_theorem} can be stated in the following stronger form:  ``If $\bar x$ is a local minimum of~{\normalfont(P)}, then~\eqref{Fermat_Rule} holds and the conditions
\begin{description}
	\item[{\rm (c1)}] $\langle \nabla f(\bar x), w \rangle \ge 0$ for all $w \in T_C^2(\bar x, v)$, where $v\in T_C(\bar x)$ is such that $\langle \nabla f(\bar x), v\rangle =0$ (i.e., $v$ is a \textit{critical direction}), 
	\item[{\rm (c2)}] $\langle \nabla^2 f(\bar x) v,v \rangle \ge 0$ for all $v\in T_C(\bar x)$ satisfying $\langle \nabla f(\bar x), v\rangle =0$
\end{description}
are fulfilled.''?}

\medskip
If $C$ is a generalized polyhedral convex set, we can answer the above question as follows.
	 
	\begin{theorem}\label{polyhedral_case}
	Let $C$ be a generalized polyhedral convex set in a Banach space~$X$. If $\bar x$ is a local minimum of~{\normalfont(P)}, then~\eqref{Fermat_Rule} holds and the conditions {\rm (c1)} and~{\rm (c2)} are fulfilled. 
	\end{theorem}
	\noindent\begin{proof} To obtain (c1), pick an arbitrary vector
	$w \in T_C^2(\bar x, v)$, where $v\in T_C(\bar x)$ and $\langle \nabla f(\bar x), v\rangle =0$. Applying Lemma~\ref{Lemma1}, we have  $\langle \nabla f(\bar x), w \rangle \ge 0$. 
	
	To prove (c2), take any $v\in T_C(\bar x)$ with $ \langle \nabla f(\bar x) , v \rangle =0$. If $v=0$, then the inequality $\langle \nabla^2 f(\bar x) v,v \rangle \ge 0$ is obvious. Now, assume that $v\neq 0$. On one hand, since $C$ is a generalized polyhedral convex set, Proposition~2.22 from~\cite{Luan_Yao_Yen} guarantees that \begin{align*}
		T_C(\bar x)={\rm cone}\, (C-x)=\{ \lambda(x-\bar x) \mid \lambda > 0, \ x\in C \}.
		\end{align*}
		Hence, we have $v=\lambda_0(y-\bar{x})$ for some $y\in C$, $y\not= \bar x$, and $\lambda_0>0$. On the other hand, as $\bar x$ is a local minimum of~{\normalfont(P)}, there exists $\varepsilon >0$ such that $f(\bar x)\le f(x)$ for every $x\in C$ with $||x-\bar x|| \le \varepsilon.$ Put $ \bar \lambda =\min\{\lambda_0, \varepsilon (\lambda_0||y-\bar x|| )^{-1}\}$. Then, $\bar \lambda>0$ and we have $\bar x + \lambda v \in C$ and $||(\bar x +\lambda v)-\bar x|| \le \varepsilon$ for all $\lambda \in (0, \bar \lambda]$.
		Therefore,
		\begin{align*}
		f(\bar x)\le f(\bar x +\lambda v) &=f(\bar x) +\lambda \langle \nabla f(\bar x), v \rangle + \frac{\lambda^2}{2}\langle \nabla^2 f(\bar x)v,v \rangle + o(\lambda^2)\\
		&=f(\bar x) + \frac{\lambda^2}{2}\langle \nabla^2 f(\bar x)v,v \rangle + o(\lambda^2).
		\end{align*}
		It follows that $\frac{\lambda^2}{2}\langle \nabla^2 f(\bar x)v,v \rangle + o(\lambda^2)\geq 0$ for all $\lambda \in (0, \bar \lambda]$. Dividing both sides of the last inequality by $\frac{\lambda^2}{2}$ and taking the limit as $\lambda\to 0^+$, we get $\langle \nabla^2 f(\bar x)v,v \rangle \ge 0,$ as desired. $\hfill\Box$
	\end{proof}

    \begin{remark}\label{remark_2} {\rm In the setting of Theorem~\ref{polyhedral_case}, one has $T_C(\bar x)\subset T^2_C(\ox,v)$ for any  $v\in T_C(\bar x)$. Since the inclusion of sets can be strict (see Remark~\ref{remark_1}), the property~(c1) asserted by Theorem~\ref{polyhedral_case} is more stringent than the first-order necessary condition in~\eqref{Fermat_Rule} which reads as follows: $\langle \nabla f(\ox),u\rangle \ge 0$ for every $u\in T_C(\ox)$.}
	\end{remark}

	As an application of Theorem~\ref{polyhedral_case}, we now specialize it to the case of quadratic programming problems on Banach spaces with generalized polyhedral convex constraint sets. Note that the later problems have been considered, for example, in \cite{Bonnans_Shapiro_2000} and \cite{Yen_Yang_2018}. One calls (P) a \textit{quadratic programming problem} on a generalized polyhedral convex set if $C\subset X$ is a generalized polyhedral convex set and $f(x)=\frac{1}{2}\langle Mx,x\rangle + \langle q, x \rangle+\alpha$, where $M:X\to X^*$ is a bounded linear operator, $q\in X^*$, and $\alpha\in\mathbb R$. It is assumed that $M$ is \textit{symmetric} in the sense that $\langle Mx,y\rangle = \langle My,x\rangle$ for all $x,y\in X$. Since $\nabla f(x)=Mx+q$ and $\nabla^2 f(x) v=Mv$ for all $x,v\in X$, the next statement follows directly from Theorem~\ref{polyhedral_case}.
	
	 \begin{theorem}\label{Theorem_QP} Assume that {\rm (P)} be a quadratic programming problem given by  a generalized polyhedral convex set $C\subset X$ and a linear-quadratic function  $f(x)=\frac{1}{2}\langle Mx,x\rangle + \langle q, x \rangle+\alpha$ with $M$ being symmetric. If $\bar x$ is a local minimum of this problem~{\normalfont(P)}, then the following conditions are satisfied:
	 	\begin{description}
	 			\item[{\rm (c0)}] $\langle M\bar x+q, v   \rangle \ge 0$ for all $v\in T_C(\ox)$;
	 		\item[{\rm (c1')}] $\langle M\bar x+q, w \rangle \ge 0$ for all $w \in T_C^2(\bar x, v)$, where $v\in T_C(\bar x)$ is such that $\langle M\bar x+q, v\rangle =0$, 
	 		\item[{\rm (c2')}] $\langle Mv,v \rangle \ge 0$ for all $v\in T_C(\bar x)$ satisfying $\langle M\bar x+q, v\rangle =0$.
	 	\end{description}
	 	\end{theorem}

According to the Majthay-Contesse theorem (see~\cite[Theorem~3.4]{Lee_Tam_Yen}), second-order necessary optimality conditions for finite-di\-men\-sio\-nal quadratic programs are also sufficient ones. Thus, it is of interest to know whether a similar assertion remains true for the second-order necessary optimality conditions in Theorem~\ref{Theorem_QP}, or not.

\medskip
\noindent {\sc Question 2:} \textit{Under the assumptions of Theorem~\ref{Theorem_QP}, if $\bar x\in C$ is such that the conditions {\rm (c0)}, {\rm (c1')}, and {\rm (c2')} are fulfilled, then $\bar x$ is a local minimum of~{\normalfont(P)}?}

\medskip
Turning our attention back to Theorem~\ref{polyhedral_case}, observe that if $C$  is not a generalized polyhedral convex set, then the assertions of that theorem may not hold anymore. This means that, in general, the pair of conditions (c1) and (c2) is much stronger than condition \eqref{differentiable_case}. 

\medskip To clarify the above observation, we first consider an example where $C$ is a compact convex set in $\mathbb R^2$, which is given by a simple inequality.

\begin{example}{\rm (See \cite[Example 2, p.~20]{LVD2014}) Consider problem (P) where $X=\mathbb R^2$, $f(x)=-2x_1^2-x_2^2$ for all $x=(x_1,x_2)$, and $$C=\big\{x=(x_1,x_2)\mid g(x)=2x_1^2+3x_2^2-6 \le 0\big\}.$$
 Since $f$ is continuous and $C$ is compact, (P) has a global solution. As $f$ is Fr\'echet differentiable, by a well known necessary optimality condition (see the proof of Theorem~5.1 in \cite{Mordukhovich_2006b}) which is a dual form of the condition recalled in Theorem~\ref{necessary_condition}, if $\bar x=(\bar x_1, \bar x_2)$ is a solution of (P) then 
 \begin{align}\label{Fermat_Ex}
 0\in \nabla f(\bar x)+\widehat{N}_C(\bar x).
 \end{align}
On one hand, $\nabla f(\bar x)=(-4\bar x_1, -2\bar x_2)^T$. On the other hand, as $C$ is a convex set, $\widehat{N}_C(\bar x)$ coincides with the normal cone to $C$ at $\bar x$ in the sense of convex analysis. Hence, by~\cite[p.~206]{IoffeTihomirov} we have 
$\widehat{N}_C(\bar x)=\{\lambda \nabla g(\bar x)=\lambda(4 \bar x_1, 6\bar x_2)^T\mid \lambda \ge 0\}$ whenever $\bar x\in \partial C$. Therefore, if $\bar x\in \partial C$, then \eqref{Fermat_Ex} is equivalent to the existence of $\lambda \ge 0$ satisfying
$$\begin{cases}
-4 \bar x_1 + 4 \lambda \bar x_1 =0\\
	-2 \bar x_2 + 6\lambda \bar x_2 =0.
\end{cases}
$$
From this condition, we get four critical points $\bar x ^1=(\sqrt{3}, 0)^T$, $\bar x ^2=(-\sqrt{3}, 0)^T$, $\bar x ^3=(0, -\sqrt{2})^T$, $\bar x ^4=(0, \sqrt{2})^T$. If $\bar x\in {\rm int}C$, then~\eqref{Fermat_Ex} is equivalent to the condition $\nabla f(\bar x)=0$, which gives the fifth critical point $\bar x ^5=(0, 0)^T$. Comparing the values of $f$ at these five points, we conclude that $\bar x ^1=(\sqrt{3}, 0)^T$ and $\bar x ^2=(-\sqrt{3}, 0)^T$ are the global minima of (P). Obviously, there exists $x^0\in\mathbb R^2$ such that $\langle \nabla g(\bar x^1), x^0\rangle< 0.$ This means that the regularity condition in \cite[Lemma 3.16]{Ruszczynski2006} is satisfied. So, according to \cite[formula~(3.29), p.~115]{Ruszczynski2006}, one has
\begin{align*}
T_C(\bar x^1)&=\{v \in \mathbb{R}^2 \mid \langle\nabla g(\bar x^1), v \rangle \le 0 \}\\
&=\{v=(v_1, v_2)\in \mathbb{R}^2 \mid v_1 \le 0,\ v_2 \in \mathbb{R} \}.
\end{align*}
Since $\nabla f(\bar x^1)=\left( -4 \sqrt{3},0 \right)^T$, fixing any $v=(0,v_2)^T \in T_C(\bar x^1)$, we have $\langle \nabla f(\bar x^1), v \rangle =0$. Moreover, by~\cite[Lemma 3.44]{Ruszczynski2006}, 
\begin{align*}
T^2_C(\bar x^1, v)&=\{w=(w_1,w_2) \in \mathbb{R}^2 \mid \langle \nabla g (\bar x^1), w \rangle \le - \langle \nabla^2 g(\bar x^1) v,v \rangle \}\\
&=\Big\{w=(w_1,w_2) \in \mathbb{R}^2 \mid w_1 \le \dfrac{-6 v_2^2}{4 \sqrt{3}}\Big\}.
\end{align*}
It follows that $\langle \nabla f(\bar x^1), w \rangle=-4 \sqrt{3}w_1\geq 0$ for every $w \in T_C^2(\bar x, v)$. Hence, condition (c1) in Theorem \ref{polyhedral_case} is satisfied. Since $\langle \nabla^2 f(\bar x^1) v, v \rangle =-2v_2^2$, the requirement $\langle \nabla^2 f(\bar x) v,v \rangle \ge 0$ in condition (c2) is violated if $v_2\neq 0$. Thus, the pair of conditions (c1) and (c2) does not hold, while condition~\eqref{differentiable_case} is fulfilled.}
\end{example}

Next, let us consider an example where $C$ is a nonconvex compact set given by an equality.

\begin{example}(See \cite[Example 1, p.~29]{LVD2014}) \rm 
Consider problem {\normalfont(P)} and suppose that $f(x)=-x_1^2-x_2^2$ for $x=(x_1,x_2)\in\mathbb R^2$, $$C=\big\{x=(x_1,x_2)\in\mathbb R^2 \mid h(x)=x_1^2+2x_2^2-1 = 0\big\}.$$
As it has been shown in~\cite[p.~29]{LVD2014}, $\bar x ^1=(1, 0)^T$ and $\bar x ^2=(-1, 0)^T$ are the global solutions of this problem. According to \cite[Formula~(3.29), p.~115]{Ruszczynski2006},
\begin{align*}
T_C(\bar x^2)=\{v=(v_1, v_2)\in\mathbb{R}^2 \mid v_1 = 0\}.
\end{align*}
Fixing any $v=(0, v_2)^T\in T_C(\bar x^2)$, we have $\langle\nabla f(\bar x^2),v\rangle=0$. By~\cite[Lemma~3.44]{Ruszczynski2006},
\begin{align*}
T^2_C(\bar x^2, v)&=\{w=(w_1,w_2) \in \mathbb{R}^2 \mid \langle \nabla h (\bar x^2), w \rangle = - \langle  \nabla^2 h(\bar x^2) v,v \rangle \}\\
&=\{w=(w_1,w_2) \in \mathbb{R}^2 \mid w_1=2v_2^2 \}.
\end{align*}
Since $\langle \nabla f(\bar x^2), w \rangle=2w_1=4v_2^2 \ge 0$ for all $w \in T_C^2(\bar x, v)$, condition~(c1) in Theorem~\ref{polyhedral_case} is satisfied. Meanwhile, since $\langle \nabla^2 f(\bar x^2) v ,v\rangle =-2v_2^2 \le 0$, the inequality $\langle \nabla^2 f(\bar x) v,v \rangle \ge 0$ in condition (c2) is violated if $v_2\neq 0$. Thus, the conditions (c1) and (c2) do not hold simultaneously, while condition~\eqref{differentiable_case} is fulfilled.
\end{example}

\section{Problems in a new setting}
\markboth{\centerline{\it Optimality conditions}}{\centerline{\it D.T.V.~An
		and N.D.~Yen}} \setcounter{equation}{0}
	
The following second-order necessary optimality condition for (P) is one of the main results of this paper. It is based on the Fr\'echet second-order subdifferential of $f$ and the second-order tangent set to $C$, which is assumed to be a convex set of a special type. Unlike the situation in Theorem~\ref{polyhedral_case} where $f$ was assumed to be a $C^2$-smooth function, in the next theorem and throughout this section we just assume that $f$ is a $C^1$-smooth function.

\begin{theorem}{\rm (Second-order necessary optimality condition)}\label{second_order_condition}
	Assume that $\ox$ is a locally optimal solution of {\normalfont(P)}, where $C$ is a generalized polyhedral convex set. Suppose that there exists a constant $\ell>0$ such that
\begin{align}\label{calmness_of_f}
||\nabla f(x) -\nabla f(\bar x)||\le \ell||x - \bar x||
\end{align}
for every $x$ in some neighborhood of $\bar x$. Consider the restricted second-order subdifferential $\widehat{\partial}^2 f(\ox): X \rightrightarrows X^*$, where $X$ is canonically embedded in $X^{**}$. Then, \eqref{Fermat_Rule} is valid and, for each $v\in T_C(\ox)$ such that $-v\in T_C(\ox)$ and $\langle\nabla f(\ox) , v \rangle =0$, one has
\begin{align}\label{second_order_formula_1n}
\langle\nabla f(\ox), w \rangle\ge 0
\end{align}
and 
\begin{align}\label{second_order_formula_1na}
\langle z,v \rangle \ge 0
 \end{align} 
for any $w\in T_C^2(\ox,v)$ and $z\in \widehat{\partial}^2 f(\ox)(v).$
\end{theorem}
\noindent\proof\  Let $\ox$ be such a locally optimal solution of {\normalfont(P)} that 
\eqref{calmness_of_f} is valid for all $x$ in a neighborhood $U$ of $\bar x$, where  $\ell$ is a positive constant.
Let $v\in T_C(\ox)$ be such that  $-v\in T_C(\ox)$ and $\langle\nabla f(\ox) , v \rangle =0$. Suppose that $w \in T_C^2(\bar x, v) $ and $z\in \widehat{\partial}^2 f(\ox)(v)$ are given arbitrarily. Since $C$ is a generalized polyhedral convex set, by Lemma~\ref{Lemma1} we have~\eqref{second_order_formula_1n}. It remains to prove~\eqref{second_order_formula_1na}. To obtain a contraction, suppose that 
	\begin{align}
	\label{assumption_contrary}
	\langle z,v \rangle < 0.
	\end{align}
By the definition of Fr\'echet second-order subdifferential, from $z\in \widehat{\partial}^2 f(\ox)(v)$ we get $ z \in \widehat{D}^{*} \nabla f(\cdot)(\bar{x})(v)$ or, equivalently, $ (z,-v) \in \widehat{N}_{\textrm{gph} \nabla f(\cdot)}((\bar{x},\nabla f(\ox))).$ So, one has
	\begin{align}\label{formula_1a}
	\limsup\limits_{x\rightarrow \bar x}\frac{\langle
		(z,-v),\big(x,\nabla f(x)\big)-(\bar x,\nabla f(\ox))\rangle}{\|x-\bar x\|+\|\nabla f(x)-\nabla f(\ox)\|} \leq 0.\end{align} Recall that every vector $u\in X$ can be regarded as an element of $X^{**}$ by setting $\langle u,x^*\rangle =\langle x^*,u\rangle$ for all $x^*\in X^*$. Hence $\langle u,\nabla f(x)\rangle=\langle \nabla f(x),u\rangle$ for all $u, x\in X$.   
	Since $\langle \nabla f(\ox), v \rangle =0$, from \eqref{formula_1a} we obtain
	\begin{align}\label{formula2}
	\limsup\limits_{x\rightarrow \bar x}\frac{\langle z, x-\bar{x} \rangle-\langle \nabla f(x),v \rangle}{\|x-\bar x\|+\|\nabla f(x)-\nabla f(\ox)\|} \leq 0.
	\end{align}
	Moreover, as $C$ is a generalized polyhedral convex set, there exists $\bar k\in\mathbb N$ such that $x^k:=\bar x-\frac{1}{k}v$ belongs to $C$ for all $k\geq \bar k$.
		
	{Since $\bar{x} $ is a local solution of~{\normalfont (P) and $\displaystyle\lim_{k\to\infty}x^k=\bar x$}, there is no loss of generality in assuming that \begin{align}\label{ineq_k}f(x^k)\geq f(\ox),\ \, \forall k\geq \bar k.\end{align} For each $k\geq \bar k$, by the classical mean value theorem one can find a vector $$\xi^k\in (\bar x,x^k):=\{(1-\tau)\bar
		x+\tau x^k\ |\ \tau\in (0,1)\}$$ such that
		$f(x^k)-f(\ox) = \langle \nabla f(\xi^{k}), x^k -\ox \rangle.$ Since $x^k=\bar x-\frac{1}{k}v$, combining this  with~\eqref{ineq_k} yields $-\frac{1}{k}\langle \nabla f(\xi_{k}), v \rangle\geq 0.$  It follows that
		\begin{align}
	\label{formula3}
\langle \nabla f(\xi_{k}), v \rangle \le 0\quad (\forall k\geq\bar k).
	\end{align}
	From \eqref{formula2} we can deduce that \begin{align*}
	\limsup\limits_{{k} \rightarrow \infty}\frac{\langle z, \xi_{k}-\bar{x} \rangle-\langle \nabla f(\xi_{k}),v \rangle}{\|\xi_{k}-\bar x\|+\|\nabla f(\xi_{k})-\nabla f(\ox)\|} \leq 0.
	\end{align*} Noting that $\xi_{k}=\ox -t_kv$ for some $t_k \in \left(0, \frac{1}{k}\right)$, from this one gets 
	\begin{align}\label{formula4}
	\limsup\limits_{{k} \rightarrow \infty}\Delta_k \leq 0,
\end{align} where  $$\Delta_k:= \frac{-t_k\langle z, v \rangle-\langle \nabla f(\xi_{k}),v \rangle}{\|-t_k v||+\|\nabla f(\xi_{k})-\nabla f(\ox)\|}.$$ Clearly,
\begin{align*}
\Delta_k= \frac{-\langle z, v \rangle-t_k^{-1}\langle \nabla f(\xi_{k}),v \rangle}{\|v||+t_k^{-1}\|\nabla f(\xi_{k})-\nabla f(\ox)\|}.\end{align*} Hence, by~\eqref{formula3} one has
$$\Delta_k \ge \frac{-\langle z, v \rangle}{\|v||+t_k^{-1}\|\nabla f(\xi_{k})-\nabla f(\ox)\|}.$$ On one hand, using~\eqref{calmness_of_f} we obtain
	\begin{align*}
	|| \nabla f(\xi_{k}) -\nabla f(\ox)||\le \ell || \xi_{k} - \ox ||=\ell t_k|| v||,
	\end{align*} provided that $k$ is large enough. On the other hand, by virtue of~\eqref{assumption_contrary} we have $-\langle z, v \rangle>0$. Consequently, for large enough indexes $k$, it holds that
	$$ \Delta_k \ge \frac{-\langle z, v \rangle }{(1+\ell) \|v\|}.$$ 
	So, we get $ \limsup\limits_{{k} \rightarrow \infty} \Delta_k > 0$, which contradicts~\eqref{formula4}. 
	
	The proof is complete.
	$\hfill\Box$

\begin{remark}\label{remark_3}{\rm To compare Theorem~\ref{second_order_condition} with Theorem~\ref{polyhedral_case}, assume for a while that $f$ is $C^{2}$-smooth. Let $\ox$ be a locally optimal solution of {\normalfont(P)}, where $C$ is a generalized polyhedral convex set. Then, applying the mean-value theorem for vector-valued functions (see~\cite[p.~27]{IoffeTihomirov}) to the gradient mapping $\nabla f(\cdot):X\to X^*$, one can show that there exists a constant $\ell>0$ such that~\eqref{calmness_of_f} holds for every $x$ in some neighborhood of $\bar x$. Since $\widehat{\partial}^2 f(\bar x)(u)=\{\nabla^2 f(\ox)^* u\}$ for every $u$ in the space $X$, which is canonically embedded in $X^{**}$, inequality~\eqref{second_order_formula_1na} means that $\langle \nabla^2 f(\ox)^* v,v \rangle \ge 0$. Hence, $\langle  v,\nabla^2 f(\ox)v \rangle \ge 0$. By the definition of the canonical embedding of $X$ in $X^{**}$, the latter means that $\langle  \nabla^2 f(\ox)v,v \rangle \ge 0$. Therefore, the assertions of Theorem~\ref{second_order_condition} coincide with those of Theorem~\ref{polyhedral_case}, provided that the critical direction $v$ satisfies the condition $-v\in T_C(\ox)$. Thus, in comparison with Theorem~\ref{polyhedral_case}, although Theorem~\ref{second_order_condition} helps us to treat optimization problems with objective functions from a larger class, it does not provide a complete extension for the former theorem.}
\end{remark}

When $C=X$, (P) becomes the unconstrained optimization problem
\begin{equation*}
\min \{f(x) \mid x \in X\} \tag{\rm P1}
\end{equation*}
with $ f: X\rightarrow \mathbb{R} $ being a $ C^{1}$-smooth function. From Theorem~\ref {second_order_condition} one can easily derive the following second-order optimality condition for (P1), which is due to Chieu \textit{et al.}~\cite{ChieuLeeYen2017}.

\begin{theorem}{\rm (See \cite[Theorem 3.3]{ChieuLeeYen2017})}
	\label{CLY2013} Suppose that $\ox$ is a local solution of {\normalfont(P1)} and there exists $\ell >0$ such that $||\nabla f(x)-\nabla f(\ox)|| \le \ell ||x-\ox||$
	for every $x$ in some neighborhood of $\ox$. Then $\nabla f(\ox)=0$ and the second-order subdifferential $\widehat{\partial}^2 f(\ox): X \rightrightarrows X^*$, where $X$ is canonically embedded in $X^{**}$, is positive semi-definite, i.e., $\langle z, u \rangle \ge 0$ for any $u\in X$ and $z\in \widehat{\partial}^2 f(\ox) (u).$
\end{theorem}

 Dai~\cite[Chapter~3]{LVD2014} has extended the finite-dimensional version of Theorem~\ref{CLY2013} to case of constrained $C^1$-smooth optimization problems of the form 
\begin{equation*}
\min \{f(x) \mid h(x)=0\}\tag{\rm P2}
\end{equation*}
with $h(x)=Ax+b$, where $A\in \mathbb{R}^{p\times n}$ is a given matrix and $b\in \mathbb{R}^p$ is a given vector. In this case, one has $C=\{x\in\mathbb R^n\mid Ax+b=0\}$. Thus, $C$ is a special polyhedral convex set in $\mathbb R^n$. The Lagrange function associated with {\rm (P2)} is defined by setting $ L(x, \mu)=f(x)+ \langle \mu, h(x)\rangle$ for $(x,\mu)\in \mathbb R^n\times\mathbb R^p$.

\medskip
\begin{theorem}{\rm (See \cite[Theorem 3.3]{LVD2014})}
	\label{LVD2014}
	Suppose that $\ox$ is a local solution of {\normalfont(P2)} and $\bar \mu\in\mathbb R^p$ is a Lagrange multiplier corresponding to $\ox$, that is, \begin{equation}\label{1st_order_Langrange}\nabla_x L (\ox, \bar\mu)=\nabla f(\bar x)+A^T\bar\mu=0.\end{equation} Suppose that, in addition, there exists a constant $\ell>0$ and a neighborhood $U$ of $\ox$ such that $||\nabla f(x)-\nabla f(\ox)|| \le \ell ||x-\ox||$
	for all $x\in U$. Then, for any $v\in \mathbb{R}^n$ with $Av=0$, one has $\langle z, v \rangle \ge 0$ for any $z\in \widehat{\partial}^2 L(\cdot,\bar\mu)(\ox)(v)$.
\end{theorem}

Theorem~\ref{second_order_condition} is a generalization of Theorem~\ref{LVD2014}. Indeed, the existence of~$\bar \mu\in\mathbb R^p$ satisfying~\eqref{1st_order_Langrange} follows from the necessary condition in~\eqref{Fermat_Rule} and Farkas' Lemma (see, e.g., \cite[p.~200]{Rockafellar_1970}). On one hand, since $\nabla_x L(x, \mu)=\nabla f(x)+ A^T\mu$ for every $(x,\mu)\in \mathbb R^n\times\mathbb R^p$, one has $\widehat{\partial}^2 L(\cdot,\bar\mu)(\ox)(\cdot)=\widehat{\partial}^2 f(\bar x)(\cdot)$. Hence, the inclusion $z\in \widehat{\partial}^2 L(\cdot,\bar\mu)(\ox)(v)$ is equivalent to saying that $z\in \widehat{\partial}^2 f(\bar x)(v)$. On the other hand, as $T_C(\bar x)=\{u\in\mathbb R^n\mid Au=0\}$, the condition $Av=0$ implies that $v\in T_C(\ox)$ and $-v\in T_C(\ox)$. Moreover, from~\eqref{Fermat_Rule} one deduces that $\langle\nabla f(\ox) , v \rangle =0$. Therefore, its follows from~\eqref{second_order_formula_1na} that $\langle z, v \rangle \ge 0$ for any $z\in \widehat{\partial}^2 L(\cdot,\bar\mu)(\ox)(v)$.

\medskip
Theorem~\ref{second_order_condition} asserts that inequality~\eqref{second_order_formula_1na} holds 
for any $z\in \widehat{\partial}^2 f(\ox)(v)$ if the critical direction $v$ satisfies the \textit{additional condition} $-v\in T_C(\ox)$.  The following example will show that \textit{the last condition is essential} for the validity of the assertion. 
	
\begin{example} Let $n=1$, $C=\mathbb R_+$, $g(x)=-x$ for $x\leq 0$ and  $g(x)=x^2$ for $x\geq 0$. Define $f(x)=\displaystyle\int_0^x g(t)dt$ for all $x\in\mathbb R$, where the integration is Riemannian. Since $g(\cdot)$ is continuous on $\mathbb{R}$, {$f$} is a $C^1$-smooth function and $\nabla f(x)=g(x)$ for $x\in\mathbb R$. Note that  $f(x)=-\frac{1}{2}x^2$ for $x\leq 0$, $f(x)=\frac{1}{3}x^3$ for $x\geq 0$. Consider the point $\bar x:=0$, which is the unique global solution of {\normalfont(P)}. Clearly, $f$ satisfies condition~\eqref{calmness_of_f} for every $x\in (-1,1)$ with $\ell=1$. On one hand, by Proposition~\ref{second-order-tangent-polyhedral_1} we have $T_C(\bar x)= \mathbb{R_+}$ and \begin{align*}
		T^2_C(\bar x, v)=T_{T_C(\bar x)}(v)=\begin{cases}
		\mathbb{R} & \mbox{if} \  v >0,\\
		\mathbb{R_+} & \mbox{if} \ v=0.
		\end{cases}
		\end{align*}
		On the other hand, using the definition of the second-order subdifferential, we have	
		\begin{align*}\begin{array}{rcl}
		z\in\widehat{\partial}^{2} f(\bar{x})(v)   &\Leftrightarrow & z \in \widehat{D}^*\nabla f(\cdot)(\bar{x})(v)  \\
		&\ \Leftrightarrow\ &  (z,-v)  \in \widehat{N}_{\textrm{gph} \nabla f(\cdot)} ((\bar{x}, \nabla f (\bar x))) \\
		&\Leftrightarrow & \limsup\limits_{x\to\ \bar{x}} \dfrac{\langle(z, -v), (x,\nabla f(x))-(\bar{x},\nabla f (\bar x))\rangle}{|x-\bar{x}|+|\nabla f(x)-\nabla f(\bar{x})|} \le 0.
		\end{array}
		\end{align*}
		Since $\bar x=0$ and $\nabla f(\bar x)=0$, the last inequality is equivalent to
		\begin{align}\label{2.1c}
		\limsup\limits_{x\to 0}  \dfrac{z x-v\nabla f(x)}{|x|+|\nabla f(x)|}  \leq0.
		\end{align} 
		From~\eqref{2.1c} one has
		\begin{align*}
		0 \ge	\limsup\limits_{x\to 0^+}  \dfrac{z x-vx^2}{x+x^2} = \limsup\limits_{x\to 0^+}  \dfrac{z -vx}{1+x}=z
		\end{align*}
		and
		\begin{align*}
		0 \ge	\limsup\limits_{x\to 0^-}  \dfrac{z x+vx}{-2x} = \dfrac{-(z+v)}{2}.
		\end{align*}
	It follows that
		\begin{align}\label{2.12c}
		z\leq 0 \quad \mbox{and}\quad z + v\ge 0.
		\end{align}
		Conversely, if~\eqref{2.12c} is satisfied, then~\eqref{2.1c} holds. Consequently, the inclusion~$z\in\widehat{\partial}^{2} f(\bar{x})(v)$ means that $-v\leq z\leq 0$. So, choosing $v=1$ and $z=-1$, one has  $v\in T_C(\ox)$, $\nabla f(\ox)v =0$, and $z\in \widehat{\partial}^2 f(\ox)(v)$. Clearly, \eqref{second_order_formula_1n} holds for any $w\in T_C^2(\ox,v)$ because $\nabla f (\bar x)=0$. However, \eqref{second_order_formula_1na} is violated as $zv=-1$. Note that $-v\notin T_C(\ox)$.
	\end{example}

{\footnotesize \noindent \textbf{Acknowledgements.} This research was supported by  Vietnam Institute for Advanced Study in Mathematics (VIASM). Duong Thi Viet An was also supported by the Simons Foundation Grant Targeted for Institute of Mathematics, Vietnam Academy of Science and Technology.}


\begin{thebibliography}{20}
\bibitem{Ban_Mordukhovich_Song_2011} Ban, L., Mordukhovich, B.S., Song, W.:	Lipschitzian stability of parametric variational inequalities over generalized polyhedra in Banach spaces. Nonlinear Anal. 74, 441--461 (2011)
\bibitem{Ben-Tal1980} Ben-Tal, A.: Second-order and related extremality conditions in nonlinear programming.  J. Optim. Theory Appl. 31, 143--165 (1980)
\bibitem{Ben-Tal1982} Ben-Tal, A., Zowe, J.: Necessary and sufficient optimality conditions for a class of nonsmooth minimization problems. Math. Programming 24, 70--91 (1982)
\bibitem{Bonnans_Shapiro_2000} Bonnans, J.F., Shapiro, A.: Perturbation Analysis of	Optimization Problems. Springer, New York (2000)
\bibitem{ChieuChuongYaoYen2011} Chieu, N.H., Chuong, T.D., Yao, J.-C., Yen, N.D.: Characterizing convexity of a function by its Fr\'echet and limiting second-order subdifferentials. Set-Valued Var. Anal. 19, 75--96 (2011)
\bibitem{ChieuHuy2011} Chieu, N.H., Huy, N.Q.: Second-order subdifferentials and convexity of real-valued functions. Nonlinear Anal. 74, 154--160  (2011)
\bibitem{ChieuLeeYen2017} Chieu, N.H., Lee, G.M., Yen, N.D.: Second-order subdifferentials and optimality conditions for $ C^{1} $-smooth optimization problems. Appl. Anal. Optim. 1, 461--476 (2017)
\bibitem{LVD2014} Dai, L.V.: Necessary and Sufficient Optimality Conditions with Lagrange Multipliers. Undergraduate Thesis, University of Science, Vietnam National University (2014)
\bibitem{HSN_1984} Hiriart-Urruty, J.-B., Strodiot, J.-J., Nguyen, V.H: Generalized Hessian matrix and second-order optimality conditions for problems with $C^{1,1}$ data. Appl. Math. Optim. 11, 43--56 (1984)
\bibitem{IoffeTihomirov} Ioffe, A.D., Tihomirov, V.M.: Theory of Extremal Problems. Amsterdam, North-Holland (1979)
\bibitem{Huy_Tuyen} Huy, N.Q., Tuyen, N.V.: New second-order optimality conditions for a class of differentiable optimization problems. J. Optim. Theory Appl. 171, 27--44 (2016)
\bibitem{Lee_Tam_Yen} Lee, G.M., Tam, N.N., Yen, N.D.: Quadratic Programming and Affine Variational Inequalities. A Qualitative Study. Springer-Verlag, New York (2005)
\bibitem{Luan_Yao} Luan, N.N., Yao, J-.C.: Generalized polyhedral convex optimization problems. J. Global Optim. 75, 789--811 (2019)
\bibitem{Luan_Yao_Yen} Luan, N.N., Yao, J-.C., Yen, N.D.: On some generalized polyhedral convex constructions.  Numer. Funct. Anal. Optim. 39, 537--570 (2018)
\bibitem{Luan_Yen} Luan, N.N., Yen, N.D.: A representation of generalized convex polyhedra and applications. Optimization 69, 471--492  (2020) 
\bibitem{L_Y_2008} Luenberger, D.G., Ye, Y.: Linear and Nonlinear Programming. Springer, New York (2008)
\bibitem{McCormick1967} McCormick, Garth P.: Second order conditions for constrained minima. SIAM J. Appl. Math. 15, 641--652 (1967)
\bibitem{Mordukhovich_1992}  Mordukhovich, B.S.: Sensitivity analysis in nonsmooth optimization. In: Field,~D.A., Komkov,~V.~(eds.) Theoretical Aspects of Industrial Design Field, pp. 32--46. SIAM, Philadelphia (1992)
\bibitem{Mordukhovich_2006a} Mordukhovich,  B.S.: Variational Analysis and Generalized Differentiation, Volume~I: Basic Theory. Springer, Berlin (2006)
\bibitem{Mordukhovich_2006b}  Mordukhovich, B.S.: Variational Analysis and Generalized Differentiation, Volume~II: Applications. Springer, Berlin (2006)
\bibitem{Mo_Ro_SIOPT2012} Mordukhovich, B.S., Rockafellar, R.T.: Second-order subdifferential calculus with applications to tilt stability in optimization. SIAM J. Optim. 22, 953--986 (2012)
\bibitem{MRS_SIOPT2013} Mordukhovich, B.S., Rockafellar, R.T., Sarabi, M.E.: Characterizations of full stability in constrained optimization. SIAM J. Optim. 23, 1810--1849 (2013)
\bibitem{Penot1994} Penot, J-.P.: Optimality conditions in mathematical programming and composite optimization. Math. Programming 67, 225--245 (1994)
\bibitem{Penot1999} Penot, J-.P.: Second-order conditions for optimization problems with constraints. SIAM J. Control Optim. 37, 303--318 (1999)
\bibitem{Poli_Roc_1998} Poliquin, R.A., Rockafellar, R.T.: Tilt stability of a local minimum. SIAM J. Optim. 8, 287--299 (1998)
\bibitem{Polyak} Polyak, B.T.: Introduction to Optimization. Revised version. Optimization Software, Inc., New York (2010)
\bibitem{Rockafellar_1970} Rockafellar, R.T.: Convex Analysis. Princeton University Press, Princeton, New Jersey (1970) 
\bibitem{Ruszczynski2006} Ruszczynski, A.: Nonlinear Optimization. Princeton University Press, New Jersey (2006)
\bibitem{Yen_Yang_2018} Yen, N.D., Yang, X.: Affine variational inequalities on normed spaces. J. Optim. Theory Appl. 178, 36--55 (2018)
\end{thebibliography}
\end{document}